\documentclass{amsart}

\usepackage{amsmath, amssymb, amsthm, comment}

\newtheorem{theorem}{Theorem}[section] 
\newtheorem{lemma}[theorem]{Lemma}     
\newtheorem{corollary}[theorem]{Corollary}
\newtheorem{proposition}[theorem]{Proposition}

\theoremstyle{remark}
\newtheorem*{acknow}{Acknowledgement}

\newcommand\ppair[1]{(\,#1\,)}
\newcommand\biggppair[1]{\biggl(#1\biggr)}

\DeclareMathOperator\linsp{span}

\newcommand\col{\colon}
\newcommand\sub{\subseteq}

\newcommand{\R}{\mathsf{R}}

\newcommand{\A}{\mathrm{A}}
\newcommand{\B}{\mathrm{B}}

\newcommand{\set}[2]{\{\,{\textstyle#1};\,{\textstyle #2}\,\}}

\let\la\langle
\let\ra\rangle

\newcommand{\QG}{\mathbb{G}}
\newcommand{\bn}{\mathbb{N}}

\DeclareMathOperator\wlim{w*-lim}

\begin{document}

\title[Inclusions of TROs and conditional expectations]%
 {Inclusions of ternary rings of operators and conditional expectations}

\author{Pekka Salmi}
\address{Department of Mathematical Sciences, University of Oulu, PL 3000,
FI-90014 Oulun yliopisto, Finland}
\email{pekka.salmi@iki.fi}

\author{Adam Skalski}
\address{Institute of Mathematics of the Polish Academy of Sciences,
ul.\'Sniadeckich 8, 00-956 Warszawa, Poland}
\email{a.skalski@impan.pl}

\begin{abstract}
It is shown that if $T$ is a ternary ring of operators (TRO), $X$ is a nondegenerate sub-TRO of $T$ and there exists a contractive idempotent surjective map $P:T \to X$ then $P$ has a unique, explicitly described extension to a conditional expectation between the associated linking algebras. A version of the result for $W^*$-TROs is also presented and some applications mentioned.
\end{abstract}


\maketitle


\maketitle

Ternary rings of operators (TROs) are norm-closed subspaces of
operators acting between Hilbert spaces which are in addition stable
under the ternary product: $(a,b,c) \mapsto a b^* c$. They form a
special class of concrete operator spaces and in fact possess also an abstract characterisation in terms of the operator-space-theoretic properties (\cite{NealRusso}). A fundamental tool for their study is the construction of a so-called linking algebra, which is a particular $C^*$-algebra containing the TRO in question as a corner. As TROs share many properties with their associated linking algebras, the construction facilitates the application of operator algebraic techniques to the analysis of objects which are not algebras themselves (examples of this type can be found in \cite{effros-ozawa-ruan} and \cite{KaurRuan}).

In this note we follow this philosophy and apply it to the analysis of
a sub-TRO $X$ of a given TRO $T$, showing that if $X$ is expected,
i.e.\ there exists a contractive idempotent map from $T$ onto $X$ and
a natural nondegeneracy condition is satisfied, then the map in
question extends to a (unique, explicitly described) conditional
expectation between the respective linking algebras. The proof is based on certain matrix calculations and the formula for the norm of an element in a left (or right) linking algebra due to Hamana (\cite{hamana:triple-envelopes}). An analogous result is valid for $W^*$-TROs, i.e.\ TROs which are closed in the weak$^*$-topology, with the respective maps being normal. We indicate also some applications.

\section{Notation and preliminaries}

For  Hilbert spaces $H$ and  $K$, let $\B(H,K)$ denote the set of all bounded
operators from $H$ to $K$. A \emph{TRO} (i.e.\ \emph{ternary ring of
operators}) is a closed subspace $T\sub \B(H,K)$
such that $T T^* T \sub T$.
In this case, the norm closed linear spans $C:=\langle T T^*\rangle$
and $D:=\langle T^* T\rangle$ are $C^*$-subalgebras
of $\B(K)$ and $\B(H)$, respectively, and
these $C^*$-algebras act non-degenerately on $T$
(i.e.\ $T = C T = T D$; note also that for subsets of normed spaces we will always use angled brackets to denote closed linear spans).
The norms on these so-called \emph{left and right linking algebras}
$C$ and $D$ are determined by the module actions:
for every $c\in C$,
\begin{equation} \label{eq:mod-norm}
\|c\| = \sup\set{\|ct\|}{t\in T, \|t\|\le 1}
\end{equation}
and the analogous equation holds for every $d\in D$.
(See Hamana~\cite{hamana:triple-envelopes}, Lemma~2.3.)
The \emph{linking algebra} of $T$ is the $C^*$-algebra
\[
\A_T:= \begin{pmatrix}
         \la T T^*\ra & T \\
                  T^* & \la T^*T \ra
       \end{pmatrix} \sub \B(K\oplus H).
\]

Suppose that $T$ is a TRO and
$P\col T \to X$ is a completely contractive
projection onto a sub-TRO $X\sub T$.
As follows from the work of Youngson (\cite{youngson:cc-proj}, Corollary~3),
the map $T$ is then a \emph{TRO conditional expectation}
in the sense that
\[
\begin{split}
P(ax^*y) &= P(a)x^*y \\
P(xa^*y) &= x P(a)^* y \\
P(xy^* a) &= xy^*P(a)
\end{split}
\qquad\text{for every } a\in T,\, x,y\in X.
\]
Effros, Ozawa and Ruan showed in \cite{effros-ozawa-ruan},
Theorem~2.5, that already a contractive projection onto a sub-TRO
is necessarily a completely contractive TRO conditional
expectation.

We say that the sub-TRO $X\sub T$ is \emph{nondegenerate} if
\[
\langle X T^* T\rangle = T \qquad \text{and}\qquad
\langle T T^* X\rangle = T.
\]
Since $\langle T T^* T \rangle = T$,
it follows from these identities that in fact
\begin{gather*}
\langle X X^* T\rangle = T \qquad
\langle T X^* X\rangle = T \qquad
\langle X T^*\rangle = \langle T T^* \rangle \qquad
\langle T^* X\rangle = \langle T^* T \rangle.
\end{gather*}
The following proposition is a straighforward consequence
of the above identities.

\begin{proposition}
A sub-TRO $X\sub T$ is nondegenerate if and
only if the $C^*$-subalgebra $\A_X$ is nondegenerate in $\A_T$.
\end{proposition}

For a map $P:T \to X$  and $t \in T$ we let $P^\dagger:T^* \to X^*$ be defined by the formula $ P^{\dagger}(t)= P(t^*)^*$. If $X\subset T$ is nondegenerate and a TRO conditional expectation from $T$ onto $X$ exists, we will say that $X$ is an \emph{expected sub-TRO} of $T$.

\section{Main result}

This section contains the main result of the paper. Observe that the construction in the proof follows (and in a sense generalises) that given in Proposition 2.1 (iv) of \cite{hamana:triple-envelopes}.

\begin{theorem} \label{thm:expectation-extension}
Suppose that $T$ is a TRO and $P\col T \to X$ is a contractive idempotent map
onto a nondegenerate sub-TRO $X\sub T$.
Then there is a $C^*$-algebra conditional expectation
from the linking algebra $\A_T$ onto the linking algebra $\A_X$
\begin{equation}
E:=\begin{pmatrix} P P^\dagger & P \\
                       P^\dagger & P^\dagger P
       \end{pmatrix}
 \col \begin{pmatrix} \langle{T T^*}\rangle & T \\
                       T^* & \langle{T^* T}\rangle
       \end{pmatrix}
       \to
      \begin{pmatrix} \langle{X X^*}\rangle & X \\
                       X^* & \langle{X^* X}\rangle
       \end{pmatrix}
 \label{formE}
 \end{equation}
where the corner maps satisfy
\[
P P^\dagger \biggl(\sum_{i1=}^n a_i x_i^*\biggr) = \sum_{i=1}^n P(a_i) x_i^*
\qquad\text{and}\qquad
P^\dagger P \biggl(\sum_{i=1}^n   x_i^*a_i\biggr) = \sum_{i=1}^n   x_i^* P(a_i)
\]
for $n \in \bn$, $a_1,\ldots,a_n\in T$ and $x_1,\ldots,x_n\in X$. The extension is unique in the
sense that if
\[
E' = \begin{pmatrix} E'_{11} & E'_{12} \\
                       E'_{21} & E'_{22}
       \end{pmatrix}
\]
is another conditional expectation from $A_T$ onto $A_X$
such that $E'_{12} = P$, then $E' = E$.
\end{theorem}

\begin{proof}
Note first that Theorem 2.5 of \cite{effros-ozawa-ruan} implies that $P:T\to X$ is a TRO conditional expectation.

By nondegeneracy, finite sums
$a = \sum_i b_i x_i^*$ where $b_i\in T$ and $x_i\in X$
are dense in $\langle TT^* \rangle$. For such $a$, define
\[
P P^\dagger (a) = \sum_i P(b_i) x_i^*
\]
(that $P P^\dagger$ is in fact well defined follows shortly).
By \eqref{eq:mod-norm}
\begin{align*}
\biggl\|\sum_i P(b_i) x_i^* \biggr\|
&= \sup \set{\|\sum_i P(b_i) x_i^* y\|}{y\in X, \|y\|\le 1} \\
&= \sup \set{\|P\bigl(\sum_i b_i x_i^* y\bigr)\|}{y\in X,  \|y\|\le 1}
\le \biggl\|\sum_i b_i x_i^* \biggr\|
\end{align*}
because $P$ is a contractive TRO conditional expectation.
It follows that $P P^\dagger$ is well defined and bounded with norm $1$.
Hence we may extend $P P^\dagger$ to a map from
$\langle T T^*\rangle$ into $\langle X X^*\rangle$, and
it is onto because on $\langle X X^*\rangle$ the map
$PP^\dagger$ is clearly the identity. Consequently, $PP^\dagger$ is a
$C^*$-algebra conditional expectation from
$\langle T T^*\rangle$ onto $\langle X X^*\rangle$.

Similarly, one can check that the map
$P^\dagger P$ is a $C^*$-algebra conditional expectation from
$\langle T^* T\rangle$ onto $\langle X^* X\rangle$.

The map
\[
E=\begin{pmatrix} P P^\dagger & P \\
                P^\dagger & P^\dagger P
   \end{pmatrix}
\col \begin{pmatrix} \langle{T T^*}\rangle & T \\
                       T^* & \langle{T^* T}\rangle
       \end{pmatrix}
       \to
      \begin{pmatrix} \langle{X X^*}\rangle & X \\
                       X^* & \langle{X^* X}\rangle
       \end{pmatrix}
\]
is a well-defined projection from $\A_T$ onto its $C^*$-subalgebra
$\A_X$. To prove that this map is a conditional expectation,
it is enough to show that it is contractive.
We shall do this using the module norm identity
\eqref{eq:mod-norm} again, this time applied to the TRO
\[
\begin{pmatrix} X \\ \langle X^*X\rangle \end{pmatrix}
\sub B(H, K\oplus H)
\]
and its  left linking algebra $\A_X$.
Let $t_i, r, s, v_j\in T$, $x_i, y_j, z\in X$ and $a\in\langle X^*X\rangle$.
Then
\begin{equation} \label{eq:E}
\begin{split}
\begin{pmatrix} P P^\dagger(\sum_i t_i x_i^*) & P(r) \\[1ex]
                P^\dagger(s^*) & P^\dagger P(\sum_j  y_j^* v_j)
\end{pmatrix}\begin{pmatrix} z \\ a \end{pmatrix}
&=\begin{pmatrix} \sum_i P(t_i) x_i^*z + P(r)a \\[1ex]
                P(s)^*z + P^\dagger P(\sum_j  y_j^* v_j)a
         \end{pmatrix} \\[1ex]
&
= \begin{pmatrix} P(\sum_i t_i x_i^*z + ra) \\[1ex]
               P^\dagger P(s^*z + \sum_j  y_j^* v_j a)
         \end{pmatrix}
\end{split}
\end{equation}
where the final step needs some explanation:
note that $P$ is a TRO conditional expectation,
$P^\dagger P$ is a $C^*$-algebra conditional expectation,
and that
\begin{equation} \label{eq:PdagP}
P(t)^* P(u) = P^\dagger P\bigl(P(t)^*u\bigr) = P^\dagger P\bigl(t^*P(u)\bigr)
\end{equation}
for every $t,u\in T$ essentially by definition.

In light of \eqref{eq:E} and \eqref{eq:mod-norm},
it is enough to show that the map
\[
\begin{pmatrix} P \\ P^\dagger P \end{pmatrix}
\]
is contractive. Using \eqref{eq:PdagP} again, we have
for every $t\in T$ and $b\in\langle T^*T\rangle$ that
\begin{align*}
\biggl \| \begin{pmatrix} P(t) \\ P^\dagger P(b) \end{pmatrix} \biggr\|^2
 &= \|P(t)^* P(t) + P^\dagger P(b)^* P^\dagger P(b)\|
 = \bigl\|P^\dagger P \bigl( P(t)^*t + P^\dagger P(b)^* b\bigr) \bigr\| \\
 &\le \biggl\| \begin{pmatrix} P(t) \\ P^\dagger P(b) \end{pmatrix}^*
          \begin{pmatrix} t \\ b \end{pmatrix} \biggr\|
 \le  \biggl\| \begin{pmatrix} P(t) \\ P^\dagger P(b) \end{pmatrix} \biggr\|
         \biggl\| \begin{pmatrix} t \\ b \end{pmatrix} \biggr\|.
\end{align*}
Hence
\[
\biggl\|\begin{pmatrix} P(t) \\ P^\dagger P(b) \end{pmatrix} \biggr\|
\le \biggl\| \begin{pmatrix} t \\ b \end{pmatrix} \biggr\|,
\]
as required.

Finally, we prove the uniqueness statement.
For $t\in T$ and $x\in X$,
\[
\begin{pmatrix} E'_{11}(tx^*) & 0 \\ 0 & 0 \end{pmatrix}
= E' \biggl( \begin{pmatrix} 0 & t \\ 0 & 0 \end{pmatrix}
           \begin{pmatrix} 0 & 0 \\ x^* & 0 \end{pmatrix} \biggr)
= \begin{pmatrix} E'_{12}(t)x^* & 0 \\ 0 & 0 \end{pmatrix}.
\]
Hence $E'_{11}(tx^*) = P(t)x^* = PP^\dagger(tx^*)$.
By nondegeneracy $E'_{11} = PP^\dagger$.
Equalities for the other entry maps follow similarly.
\end{proof}

Note that the main part of the above proof provides an alternative
direct proof of a special case of Theorem 2.5 of \cite{effros-ozawa-ruan},
namely it shows that a contractive TRO conditional expectation
onto a non-degenerate sub-TRO is necessarily completely bounded
(Theorem 2.5 of \cite{effros-ozawa-ruan} was used in the
above  proof only to deduce the algebraic properties of
the TRO-conditional expectation $P$ from its contractivity, so
the proof shows that contractivity together with the algebraic properties
imply complete contractivity).
We also have the following immediate corollary.

\begin{corollary} \label{EtoP}
Let  $T$ be a TRO and $X\sub T$ a nondegenerate sub-TRO. If $E:\A_T
\to \A_X$ is a conditional expectation \textup{(}onto $\A_X$\textup{)}
and $E$ maps $T\subset A_{T}$ into $T$, then $E$ is of the form
\eqref{formE} for $P:=E|_T$.
\end{corollary}

A direct motivation to study the construction described in Theorem \ref{thm:expectation-extension}
came from the desire to characterise \emph{contractive idempotent
  functionals} on $C_0(\QG)$, the $C^*$-algebra of continuous, vanishing
at infinity functions on a \emph{locally compact quantum group} $\QG$,
in terms of certain subspaces of $C_0(\QG)$ (which turn out to be
TROs). We refer the reader to \cite{contractiveidempotents} for the
details, noting that the explicit form of the conditional expectation
established in Theorem~\ref{thm:expectation-extension}
plays a significant role there.

On the other hand Theorem \ref{thm:expectation-extension} together
with the techniques developed in \cite{KaurRuan}, showing that a TRO
shares for example many approximation properties with its linking
algebra, allows us for example to show immediately that
Lance-nuclearity (in  the terminology of \cite{KaurRuan}) passes to
expected sub-TROs.

\begin{corollary}
Let $T$ be  a TRO and $X\subset T$ an expected nondegenerate
sub-TRO. If $T$ is Lance-nuclear then so is $X$.
\end{corollary}

\begin{proof}
This follows from Theorem \ref{thm:expectation-extension}, Theorem 6.1 of
\cite{KaurRuan} and the fact that nuclearity of $C^*$-algebras is
preserved under taking conditional expectations.
\end{proof}

\section{Extension to $W^*$-TROs}

In this section we present a version of Theorem
\ref{thm:expectation-extension} for $W^*$-TROs.
A \emph{$W^*$-TRO} is a TRO $T\sub \B(H,K)$ that is closed under the weak*
topology. We may assume without loss of generality that
$T$ is \emph{nondegenerately represented}, i.e.\ that
$\la TH\ra = K$ and $\la T^* K\ra = H$.
A $W^*$-sub-TRO $X\sub T$ is \emph{nondegenerate} if
the linear spans of $XT^*T$ and $TT^*X$ are  weak*-dense in $T$. The next proposition connects this property to $X$ itself being nondegenerately represented.

\begin{proposition} \label{propn:W*-nondeg}
Suppose that $T$ is a nondegenerately represented $W^*$-TRO
in $\B(H,K)$. A $W^*$-sub-TRO $X\sub T$ is nondegenerate
if and only if $X$ is nondegenerately represented in $\B(H,K)$.
\end{proposition}

\begin{proof}
If $X$ is nondegenerate, then the linear
spans of $XT^*T$ and $TT^*X$ are (as convex sets) so-dense in $T$, so
a simple calculation allows us to approximate elements in $ T H$
(respectively, in $T^* K$) by elements in $ X H$ (respectively, in $
X^* K$).

Conversely, if $X$ is nondegenerately represented in
$\B(H,K)$, then $(XX^*)''$ is a (nondegenerate) von
Neumann algebra on $K$ and hence contains the identity operator on
$K$. It follows that $T$ is contained in the weak*-closed linear span of
$XX^*T$. Similarly, $T$ is contained in the weak*-closed
linear span of $TX^*X$, and so $X$ is nondegenerate in $T$.
\end{proof}

We need one more lemma which may be viewed as a version of the Kaplansky theorem for nondegenerate $W^*$-sub-TROs.

\begin{lemma} \label{lem:kap}
Suppose that $T$ is a nondegenerately represented $W^*$-TRO
in $\B(H,K)$ and that $X\sub T$ is  a nondegenerate $W^*$-sub-TRO.
For every $u$ in $(TT^*)''$
there is a net $(a_\alpha)$ in $\linsp(TX^*)$
such that $(a_\alpha)$ is bounded by $\|u\|$
and $a_\alpha\to u$ weak*. Similar statements hold when
$(TT^*)''$ is replaced by
\[
(T^*T)''\qquad  \left( \text{ respectively,}
\begin{pmatrix} (TT^*)'' & T \\ T^* & (T^*T)'' \end{pmatrix}\; \right)
\]
and $\linsp(TX^*)$ is replaced by
\[
\linsp(X^*T)\qquad \left(  \text{ respectively,}
 \begin{pmatrix} \linsp(TX^*) & T \\ T^* & \linsp(X^*T) \end{pmatrix}\;\; \right).
\]
\end{lemma}

\begin{proof}
As in the proof of Proposition~\ref{propn:W*-nondeg},
the C*-algebra $\la XX^*\ra$ is nondegenerately represented
on $K$. Let $(e_\alpha)$ be a contractive approximate identity
in $\la XX^*\ra$. We may suppose that each $e_\alpha\in\linsp(XX^*)$.
Moreover, note that $e_\alpha\to 1_K$ in the weak* topology due to nondegeneracy.
When $u\in \linsp(TT^*)$, the net $(ue_\alpha)$ satisfies the
statement. The general case of $(TT^*)''$ follows by weak* approximation, using the Kaplansky theorem for the inclusion $\linsp(TT^*) \subset (TT^*)''$.

As for the final statements, let
$(f_\beta)\sub \linsp(X^*X)$ be a contractive approximate identity
for $\la X^*X\ra$.
Now we may continue as above, using
the nets $(f_\beta u)$ for $u\in \linsp(T^*T)$ and
\[
\begin{pmatrix} 1_K & 0 \\ 0 & f_\beta\end{pmatrix}
v \begin{pmatrix} e_\alpha & 0 \\ 0 & 1_H \end{pmatrix}
\quad \text{for} \quad v \in
\begin{pmatrix} \linsp(TT^*) & T \\ T^* & \linsp(T^*T) \end{pmatrix}.
\]
\end{proof}

We are ready to formulate the main theorem of this section.
Note that the result does not immediately follow
from Theorem~\ref{thm:expectation-extension} as
the $W^*$-sub-TRO $X$ need not satisfy the stronger
nondegeneracy condition of Theorem~\ref{thm:expectation-extension}
(i.e.\ nondegeneracy with respect to the norm topology).

\begin{theorem} \label{thm:W*-expectation}
Assume that $X$ and $T$ are $W^*$-TROs, $X$ is nondegenerate in $T$
\textup{(}in the $W^*$-TRO sense\textup{)}
and that $P : T\to X$ is a normal \textup{(}i.e.\ weak$^*$-continuous\textup{)}
contractive surjective idempotent map. Then $P$ extends to a normal
conditional expectation from $\R_T$ onto $\R_X$, where $\R_T = \A_T''$
and $\R_X = \A_X''$.
\end{theorem}

\begin{proof}
As before, we use Theorem 2.5 of \cite{effros-ozawa-ruan} to observe
that $P$ is a TRO conditional expectation.

Define $PP^\dagger$ on $\linsp(TX^*)$ as in the $C^*$-algebra case (see Theorem \ref{thm:expectation-extension}).
Consider a net $(a_\alpha)$ in $\linsp(TX^*)$ converging weak* to
$a\in\linsp(TX^*)$. Write
$a_\alpha = \sum_{i=0}^{n_\alpha} b_i^\alpha (x_i^\alpha)^*$
and $a = \sum_{i=0}^n b_i (x_i)^*$ where $n_{\alpha} \in \bn$, $b_i^\alpha\in T$,
$x_i^\alpha\in X$, etc..
Due to nondegeneracy, functionals $\omega_{y\xi,\eta}$
with $y\in X$, $\xi\in H$ and $\eta\in K$ are linearly dense
in $\B(K)_*$. Since $P$ is normal,
\begin{align*}
\omega_{y\xi,\eta}\bigl(PP^\dagger(a_\alpha)\bigr)
&= \sum_{i=0}^{n_\alpha} \ppair{P(b_i^\alpha) (x_i^\alpha)^*y \xi\mid \eta}
=  \biggppair{P\Bigl(\sum_{i=0}^{n_\alpha}b_i^\alpha
  (x_i^\alpha)^*y\Bigr)
      \xi\biggm| \eta}\\
&\to \biggppair{P\Bigl(\sum_{i=0}^{n}b_i (x_i)^*y\Bigr) \xi\biggm| \eta}
=  \omega_{y\xi,\eta}\bigl(PP^\dagger(a)\bigr).
\end{align*}
By approximation, we see that $PP^\dagger$
is weak* continuous on bounded subsets of $\linsp(TX^*)$
(note that, as in the proof of Theorem~\ref{thm:expectation-extension},
$PP^\dagger$ is contractive on $\linsp(TX^*)$).

Extend $PP^\dagger$ to $(TT^*)''$, which is
the weak* closure of $\linsp(TX^*)$, by defining
\[
\widetilde{PP^\dagger}(u) = \wlim PP^\dagger(a_\alpha)
\]
where $(a_\alpha)\sub \linsp(TX^*)$ is bounded
and converges to $u$ in the weak* topology
(such a net exists by Lemma~\ref{lem:kap}).
First note that the extension is well-defined:
a weak* cluster point of $(PP^\dagger(a_\alpha))$ always exists
due to boundedness, and
if $(b_\beta)$ is another net satisfying the requirements,
then $a_\alpha-b_\beta \to 0$ weak* in $\linsp(TX^*)$
(joint net with the natural direction),
and since $(a_\alpha-b_\beta)$ is bounded
and $PP^\dagger$ is weak* continuous on bounded sets,
it follows that $\wlim PP^\dagger(a_\alpha) = \wlim PP^\dagger(b_\beta)$.

Next we show that the extension is weak* continuous
on bounded sets. Let $(u_\alpha)$ be a bounded net in $(TT^*)''$
converging weak* to $u$.
For every finite set $F\sub \B(K)_*$, every natural number $n$
and every $\alpha$, choose $a_{\alpha}^{n,F}$ in $\linsp(TX^*)$ with
$\|a_{\alpha}^{n,F}\|\le \|u_\alpha\|$ such that for every $\omega\in F$
\[
\omega(a_{\alpha}^{n,F} - u_\alpha) < 1/n
\]
(note the use of Lemma~\ref{lem:kap} in the choice of $a_{\alpha}^{n,F}$).
Then $\wlim_{(n,F)} a_\alpha^{n,F} = u_\alpha$ and
$\wlim_{(\alpha,n,F)} a_\alpha^{n,F} = u$. %
By definition,
\[
\widetilde{PP^\dagger}(u_\alpha) = \wlim_{(n,F)} PP^\dagger(a_\alpha^{n,F})
\quad\text{and}\quad
\widetilde{PP^\dagger}(u) = \wlim_{(\alpha,n,F)} PP^\dagger(a_\alpha^{n,F}),
\]
and hence
\begin{align*}
\wlim_\alpha\widetilde{PP^\dagger}(u_\alpha)
&= \wlim_\alpha \wlim_{(n,F)} PP^\dagger(a_\alpha^{n,F})\\
&= \wlim_{(\alpha,n,F)} PP^\dagger(a_\alpha^{n,F})
= \widetilde{PP^\dagger}(u).
\end{align*}
Thus $\widetilde{PP^\dagger}$ is weak* continuous on bounded sets.

We next claim that $\widetilde{PP^\dagger}$ is contractive.
Let $u\in (TT^*)''$ and apply Lemma~\ref{lem:kap}
to obtain a net $(a_\alpha)$ in $\linsp(TX^*)$, bounded by $\|u\|$,
converging weak* to $u$.
For every $\omega\in \B(K)_*$,
\[
|\omega\bigl(\widetilde{PP^\dagger}(u)\bigr)|
= \lim_\alpha|\omega\bigl(PP^\dagger(a_\alpha)\bigr)|
\le \lim_\alpha \|a_\alpha\|\|\omega\|
\le \|u\| \|\omega\| .
\]
That is, $\widetilde{PP^\dagger}$ is contractive.

It is easy to check that $\widetilde{PP^\dagger}$ is a projection
onto $(XX^*)''$. Indeed, $\widetilde{PP^\dagger}$
agrees with the identity map on $\linsp(XX^*)$,
and then using the Kaplansky density theorem and the fact that
$\widetilde{PP^\dagger}$ is weak* continuous on bounded sets,
we obtain the claim.
Combined with contractivity, we see that
$\widetilde{PP^\dagger}$ is a conditional expectation
and thus in particular is positive.

By the order-theoretic characterisation of normality  (see for example
\cite[Definition~46.1 and Proposition~43.1]{conway:cop}),
positivity and weak* continuity on bounded sets
imply that $\widetilde{PP^\dagger}$ is normal.
So $\widetilde{PP^\dagger}$ is a normal conditional expectation.
A similar argument produces a normal conditional expectation
$\widetilde{P^\dagger P}$, and as in the earlier case,
we form the the matrix map from $\R_T$ onto $\R_X$.

We still need to check that the matrix map is contractive.
As in the proof of Theorem~\ref{thm:expectation-extension},
we see that the map
\[
\begin{pmatrix} P P^\dagger & P \\
                P^\dagger & P^\dagger P
   \end{pmatrix}
\col \begin{pmatrix} \linsp(T X^*) & T \\
                       T^* & \linsp(X^* T)
       \end{pmatrix}
       \to
      \begin{pmatrix} \linsp(X X^*) & X \\
                       X^* & \linsp(X^* X)
       \end{pmatrix}
\]
is contractive. Using Lemma~\ref{lem:kap} again, we may apply the same
argument as in the case of $\widetilde{PP^\dagger}$ to see that
\[
\begin{pmatrix} \widetilde{P P^\dagger} & P \\
                P^\dagger & \widetilde{P^\dagger P}
   \end{pmatrix}
\]
is contractive.
\end{proof}

\begin{acknow}
The work on this paper started during the visit
of the second-named author to the University of Oulu in May 2012.
We thank the Finnish Academy of Science and Letters, Vilho,
Yrj\"o and Kalle V\"ais\"al\"a Foundation,
for making this visit possible.
The first-named author was also supported by the Emil Aaltonen
Foundation. The second-named author was
partly supported by the National Science Centre (NCN) grant
no.~2011/01/B/ST1/05011. We thank the referee for a careful and thoughtful reading of our manuscript.
\end{acknow}


\begin{thebibliography}{EOR01}

\bibitem[Con00] {conway:cop}
 J.~B. Conway, \emph{A course in operator theory},
 American Mathematical Society, Providence, RI, 2000.

\bibitem[EOR01]{effros-ozawa-ruan}
E.~G. Effros, N.~Ozawa, and Z.-J. Ruan, \emph{On injectivity and nuclearity for
  operator spaces}, Duke Math. J. \textbf{110} (2001), 489--521.

\bibitem[Ham99]{hamana:triple-envelopes}
M.~Hamana, \emph{Triple envelopes and \v {S}ilov boundaries of operator
  spaces}, Math. J. Toyama Univ. \textbf{22} (1999), 77--93.

\bibitem[KaRu02]{KaurRuan} M.\,Kaur and Z-J.\,Ruan,
Local properties of ternary rings of operators and their linking $C^*$-algebras,
\emph{J.\,Funct.\,Anal.} \textbf{195} (2002), no.\,2, 262--305.

\bibitem[NeRu02]{NealRusso}
M.~Neal and B.~Russo,
Operator space characterizations of $C^*$-algebras and ternary rings, \emph{Pacific J.\,Math.} \textbf{209} (2003), no.~2, 339--364.

\bibitem[NSSS12]{contractiveidempotents} M.~Neufang, P.~Salmi,
  A.~Skalski and N.~Spronk, Contractive idempotents on locally compact
  quantum groups, \emph{Indiana Univ. Math. J.}, to appear,
available at \emph{arXiv:1209.6508.}


\bibitem[You83]{youngson:cc-proj}
M.~A. Youngson, \emph{Completely contractive projections on {$C^{\ast}
  $}-algebras}, Quart. J. Math. Oxford Ser. (2) \textbf{34} (1983), 507--511.

\end{thebibliography}
\end{document}